\documentclass[12pt,reqno]{amsart}

\usepackage{amsfonts}
\usepackage{amssymb}
\usepackage{amsmath,mathtools}
\usepackage{latexsym}
\usepackage{amscd}
\usepackage{mathrsfs}%scrfont}
\usepackage{enumitem} %http://ctan.org/pkg/enumitem
\usepackage{braket}
\usepackage[margin=1.4in]{geometry}
\usepackage{esint}
\usepackage{amsthm}
\usepackage{accents}
\usepackage{parskip}
\usepackage{comment}

\setlist{itemsep=3pt}

\allowdisplaybreaks

\newtheorem{prop}{Proposition}
\newtheorem{theo}[prop]{Theorem}
\newtheorem{lemm}[prop]{Lemma}

\theoremstyle{definition}
\newtheorem{defi}[prop]{Definition}

\numberwithin{prop}{section}
\numberwithin{equation}{section}

\newcommand{\CC}{\mathbf{C}}

\newcommand{\RR}{\mathbf{R}}
\renewcommand{\SS}{\mathbf{S}}

\let\oldmarginpar\marginpar
\renewcommand\marginpar[1]{\-\oldmarginpar[\raggedleft\footnotesize #1]%
{\raggedright\footnotesize #1}}

\DeclareMathOperator{\dist}{dist}

\newcommand{\eps}{\varepsilon}

\usepackage{graphicx}

\allowdisplaybreaks

\begin{document}

\title[Phase transitions, Minimal Hypersurfaces, Rigidity]{Double-well phase transitions are more rigid than minimal hypersurfaces}

\author{Christos Mantoulidis}
\address{Department of Mathematics, Rice University, Houston, TX 77005}
\email{christos.mantoulidis@rice.edu}

\maketitle

\begin{abstract}
	In this short note we see that double-well phase transitions exhibit more rigidity than their minimal hypersurface counterparts. 
\end{abstract}

\section{Introduction}

Fix a Riemannian manifold $(N, g_N)$, which for simplicity we assume to be closed. There is a rich and long history of parallels between the theories of:
\begin{enumerate}
	\item minimal hypersurfaces, which are hypersurfaces that are critical points of the induced area functional
		\[ A : \{ \text{closed hypersurfaces of $N$} \} \to (0, \infty) \]
		\[ A(\Sigma) := \text{Area}_{g_N}(\Sigma), \]
		and
	\item double-well phase transitions (per the van der Waals–Cahn–Hilliard theory), which are critical points of the functional
		\[ E_\eps : C^\infty(N; [-1, 1]) \to [0, \infty) \]
		\[ E_\eps(u) := \int_N \big( \tfrac12 \eps |\nabla_{g_N} u|^2 + \eps^{-1} W(u) \big) \, d\mu_{g_N}, \]
		for some phase transition width parameter $\eps > 0$ and some double-well potential $W \in C^\infty(\RR; \RR)$. (See Definition \ref{defi:double.well}.)
\end{enumerate}
The geometric analysis research community has shown a lot of interest in understanding how and when one can go from one theory to the other. In one direction, given double-well phase transitions $u = u_i\in C^\infty(N; [-1, 1])$ with $\eps = \eps_i \to 0$, one would like to argue (if possible) that the nodal sets $\{ u_i = 0 \}$ converge, e.g., in the Hausdorff sense, to a minimal hypersurface $\Sigma$ in $(N, g_N)$. We refer the interested reader to \cite{HutchinsonTonegawa, Tonegawa, TonegawaWickramasekera, Guaraco, GasparGuaraco:ac.closed}. In the opposite direction, given a minimal hypersurface $\Sigma$ in $(N, g_N)$, one would like to produce (if possible) double-well phase transitions $u = u_i\in C^\infty(N; [-1, 1])$ with $\eps = \eps_i \to 0$ whose nodal sets $\{ u_i = 0 \}$ converge to $\Sigma$, e.g., in the Hausdorff sense.  %generating double-well phase transitions out of minimal hypersurfaces, 
We refer the interested reader to \cite{Modica, Sternberg, KohnSternberg, PacardRitore, delPinoKowalczykWei:clustering, delPinoKowalczykPacardWei:multiple.ends, delPinoKowalczykWeiYang:interface, delPinoKowalczykWei:r3}. %These lists are by no means exhaustive.

In this note we show that not all minimal hypersurfaces are limits of nodal sets of double-well phase transitions. In particular, we show that in the presence ambient symmetry ($\SS^1$ symmetry, in our case) double-well phase transitions exhibit additional rigidity properties that minimal hypersurfaces don't have:

\begin{theo} \label{theo:main}
	Suppose $(M, g_M)$ is a closed Riemannian manifold and $(N, g_N) := \SS^1 \times (M, g_M)$ is its Riemannian product with a unit circle. Let $u_i : N \to [-1, 1]$ be double-well phase transitions for a potential $W$ as in Definition \ref{defi:double.well} and parameters $\eps_i \searrow 0$ and
	\[ \lim_{i \to \infty} \{ u_i = 0 \} = \{ \theta_1, \ldots, \theta_m \} \times M =: \Sigma \text{ in the Hausdorff sense}, \]
	where $\theta_1, \ldots, \theta_m \in \SS^1$ are distinct. Then:
	\begin{enumerate}
		\item[(a)] $\lim_{i \to \infty} u_i = \pm 1$ locally uniformly on $N \setminus \Sigma$, with the $\pm$ sign necessarily alternating across each component of $\Sigma$ (thus, $m$ is even). 
		\item[(b)] $\Sigma$ is invariant under a rotation by $\tfrac{2\pi}{m}$ radians along the $\SS^1$ fiber. 
	\end{enumerate}
\end{theo}

\begin{defi} \label{defi:double.well}
	A function $W \in C^\infty(\RR; \RR)$ is a double-well potential if:
	\begin{enumerate}
		\item $W(x) \geq 0$ for all $x$, with equality if and only if $x = \pm 1$.
		\item $W(x) = W(-x)$ for all $x$.
		\item $W''(\pm 1) > 0$.
		\item $W'(x) / x$ is increasing for $x \in (0, 1)$ and decreasing for $x \in (-1, 0)$.
	\end{enumerate}
\end{defi}

Theorem \ref{theo:main} has various interesting consequences:
\begin{enumerate}
	\item There exist many degenerate minimal hypersurfaces that cannot be limits of double-well phase transitions. Indeed, $\Sigma := \{ \theta_1, \ldots, \theta_{m} \} \times M$ is totally geodesic---and therefore minimal---in $N$ for any distinct $\theta_1, \ldots, \theta_{m} \in \SS^1$, but according to Theorem \ref{theo:main} it can only occur as a limiting phase transition when $m$ is even and $\Sigma$ is $\tfrac{2\pi}{m}$-rotationally symmetric along the $\SS^1$ fiber. Thus, the nondegeneracy assumption in Pacard--Ritor\'e \cite{PacardRitore} (see also Pacard \cite{Pacard}, De Philippis--Pigati \cite{DePhilippisPigati}) cannot be weakened unconditionally. There is a result in the degenerate case due to Caju--Gaspar \cite{CajuGaspar} who assume ``strong'' integrability for their degenerate minimal hypersurfaces, i.e., that all Jacobi fields must be produced from ambient isometries. By Theorem \ref{theo:main}, this assumption cannot be weakened to the standard integrability assumption, i.e., that all Jacobi fields arise from nearby minimal hypersurfaces. Indeed, each $\Sigma$ above carries Jacobi fields that equal $a_k \partial_t$ (here, $t$ is the $\SS^1$ coordinate and $a_k \in \RR$) on each $\{ \theta_k \} \times M$, $k = 1, \ldots, m$, all of which are obviously integrable, but only ``strongly'' integrable if $a_1 = \ldots = a_m$.
	\item The double-well phase transition set of min-max varifolds, which is contained in the corresponding Almgren--Pitts set (Gaspar--Guaraco \cite{GasparGuaraco:ac.closed}, see also Dey \cite{Dey:ac.ap}), can be a proper subset. Indeed, any $\{ \theta_1, \theta_2 \} \times \SS^1$, $\theta_1, \theta_2 \in \SS^1$, can be realized as an Almgren--Pitts width on $\SS^1 \times \SS^1$, but only those with $\theta_1 = - \theta_2$ will be realized as a phase transition width.
\end{enumerate}

We view this rigidity as consistent with the stronger results one has when studying double-well phase transitions versus minimal hypersurfaces. For instance:
\begin{enumerate}
	\item The author and Otis Chodosh proved in \cite{ChodoshMantoulidis:ac.3d} the multiplicity-one property of all min-max minimal hypersurfaces constructed using double-well phase transition regularization in generic closed Riemannian 3-manifolds. Subsequently, Xin Zhou \cite{Zhou:multiplicity.one} proved the multiplicity-one conclusion in dimension up to $7$ for at least one min-max hypersurface in Almgren--Pitts theory using a different regularization, but which is not characterized by the same rigidity as, e.g., Theorem \ref{theo:main}.
	\item The author and Otis Chodosh proved in \cite{ChodoshMantoulidis:pwidths} that in all closed Riemannian 2-manifolds, all phase transition min-max stationary geodesic networks give unions of closed geodesics. There is no Almgren--Pitts analog yet, but see Calabi--Cao \cite{CalabiCao}, Aiex \cite{Aiex}.
\end{enumerate}

\section{Two simpler results}

The proposition below was discussed by the author with Alessandro Pigati in a meeting in January 2022 as rudimentary evidence for the rigidity of double-well phase transitions as opposed to minimal hypersurfaces. It can be thought of as a stronger version of Theorem \ref{theo:main} when $M$ is $0$-dimensional.

\begin{prop} \label{prop:n.1}
	Fix $\eps > 0$, and a double-well phase transition $u : \SS^1 \to [-1, 1]$ with parameter $\eps$ and $m$ nodal intervals. Then:
	\begin{enumerate}
		\item[(a)] The nodal intervals are all congruent.
		\item[(b)] On each nodal interval, $u$ equals $\pm 1$ times a positive function that vanishes on its boundary and does not change from each nodal interval to the next. The $\pm$ sign alternates across each nodal point.
	\end{enumerate}
\end{prop}
\begin{proof}
	By the existence and uniqueness of second order ODE solutions with given an initial position and velocity, $u$ must be odd about each vanishing point as well as even about each critical point, and thus even about the midpoint of the nodal interval. This implies that all intervals in $\{ u > 0 \}$ and $\{ u < 0 \}$ are congruent and that $u$ is $\pm 1$ times a positive function that does not change from one interval to the next.
\end{proof}

Functions as in Proposition \ref{prop:n.1} will not exist if $m$ is odd, or if $\eps > 0$ is large. Conversely, when $m$ is even and $\eps > 0$ is sufficiently small, functions of the desired form can be produced by a reflection argument from the following lemma. 

\begin{lemm}[Model solutions] \label{lemm:model}
	Fix $\ell > 0$. For sufficiently small $\eps > 0$, depending on $\ell$, there is a double-well phase transition $u : [-\ell, \ell] \to [0, 1]$ with parameter $\eps$ so that $u(\pm \ell) = 0$ and $u > 0$ on $(-\ell, \ell)$.
\end{lemm}
\begin{proof}
	First, one minimizes $E_\eps$ on $W^{1,2}_0(-\ell, \ell)$. Using the evenness of $W$ and that $1$ is a global minimum of $W$, one may replace $u$ by $\min \{ |u|, 1 \}$ to get that the minimizer satisfies $0 \leq u \leq 1$ a.e., and thus everywhere by elliptic regularity.  Note that $u > 0$ on $(-\ell, \ell)$ unless $u \equiv 0$, e.g., by the uniqueness of second order ODE solutions with given initial position and velocity. The case $u \equiv 0$ is ruled out for minimizers by an energy comparison argument for small $\eps > 0$. 
\end{proof}

These model solutions are used as barriers in the proof of Theorem \ref{theo:main}. 

Below, we prove the $m \leq 2$ case of Theorem \ref{theo:main} as a warm-up; the general case will require a localization of the simpler argument below.

\begin{prop} \label{prop:m.2}
	Suppose $M$, $N$ are as in Theorem \ref{theo:main} and let $u_i : N \to [-1,1]$ be double-well phase transitions with parameters $\eps_i \searrow 0$ so that, for $\theta_1, \theta_2 \in \SS^1$ (not necessarily distinct),
	\[ \lim_{ i \to \infty} \{ u_i = 0 \} = \{ \theta_1, \theta_2 \} \times M =: \Sigma \text{ in the Hausdorff sense}. \]
	Then:
	\begin{enumerate}
		\item[(a)] $\theta_1$ and $\theta_2$ are distinct and, in fact, $\theta_2 = - \theta_1$. 
		\item[(b)] $\lim_{i \to \infty} u_i = \pm 1$ locally uniformly on $N \setminus \Sigma$, with alternating signs across each component of $\Sigma$.
	\end{enumerate}
\end{prop}

\begin{proof}[Proof of Proposition \ref{prop:m.2}]
	We first prove (a). Suppose, for the sake of contradiction, that $\theta_1 \neq - \theta_2$. By the Hausdorff convergence $\{ u_i = 0 \} \to \{ \theta_1, \theta_2 \} \times M$, one can rotate $\SS^1$ to guarantee that
	\[ \theta_1, \theta_2 \in \Gamma_+:= \{ z \in \SS^1 \subset \CC : \operatorname{Re} z > 0 \} \]
	and, after possibly replacing $u_i$'s with $-u_i$'s, that for all large $i$,
	\begin{equation} \label{eq:m.2.containment}
		\{ u_i \geq 0 \} \subset \Gamma_+ \times M.
	\end{equation}
	Passing to larger $i$ if necessary, let
	\[ \bar u_i : N ( = \SS^1 \times M) \to [-1, 1] \]
	be obtained by applying Lemma \ref{lemm:model} with $\ell = \tfrac12 \pi$ and $\eps = \eps_i$, defining it on $\Gamma_+$ using the arclength parametrization $t \mapsto e^{it}$, extending to $\SS^1$ with an odd reflection, and extending trivially across $M$.
	
	By Lemma \ref{lemm:gmn}, \eqref{eq:m.2.containment} implies
	\[ u_i < \bar u_i \text{ on } \{ u_i \geq 0 \} \cup ((\SS^1 \setminus \Gamma_+) \times M). \]
	This trivially extends to
	\begin{equation} \label{eq:m.2.full.inequality}
		u_i < \bar u_i \text{ on } N,
	\end{equation}
	since $u_i < 0 < \bar u_i$ in the complementary portion of $N$. Inequality \eqref{eq:m.2.full.inequality} leads to a contradiction as we may rotate $\bar u$ until it eventually touches $u_i$, violating the maximum principle. Therefore, $\theta_1 = -\theta_2$, yielding (a).
	
	In what follows, we may assume that $\theta_1 = 1 \in \SS^1 \subset \CC$, $\theta_2 = -1 \in \SS^1 \subset \CC$.
		
	We now prove (b). By Lemma \ref{lemm:gui}, $\lim_{i \to \infty} |u_i| = 1$ on $N \setminus \Sigma$. Let us suppose, for the sake of contradiction, that
	\begin{equation} \label{eq:m.2.sign}
		\lim_{i \to \infty} u_i = -1 \text{ on } N \setminus \Sigma.
	\end{equation}
	The other case is ruled out by replacing $u_i$'s with $-u_i$'s.
	
	Passing to larger $i$ if necessary, let
	\[ \breve u_i : N ( = \SS^1 \times M) \to [-1, 1] \]
	be obtained by applying Lemma \ref{lemm:model} with $\ell = \tfrac14 \pi$ and $\eps = \eps_i$ and using the arclength parametrization $t \mapsto e^{it}$, extending to $\SS^1$ with three odd reflections, and extending trivially across $M$. By repeating the argument above, \eqref{eq:m.2.sign} implies
	\[
		u_i < \breve u_i \text{ on } N,
	\]
	for large $i$. This leads to a contradiction similarly as before, yielding (b).
\end{proof}

\section{Proof of Theorem \ref{theo:main}}

In what follows, we assume $m > 2$, otherwise Proposition \ref{prop:m.2} yields the result. The proof is, in spirit, a unique continuation-type proof like the one-dimensional proof in Proposition \ref{prop:n.1}, but it relies on the maximum principle similarly to the proof of Proposition \ref{prop:m.2}, which must now be localized.

We lift our solutions and $\Sigma$ to the covering space
\[ \RR \times M \to N ( = \SS^1 \times M) \]
where, without loss of generality, we may suppose that the angles $\theta_1, \ldots, \theta_{m}$ lift to real numbers $0 < \theta_1 < \ldots < \theta_{m} < 2\pi$. We will still denote  our solutions by $u_i$.

We first prove (a). Recall that $\lim_{i \to \infty} |u_i| = 1$ away from $\Sigma$, by Lemma \ref{lemm:gui}. We may suppose, for the sake of contradiction, that 
\begin{equation} \label{eq:m.limit.neg1}
	\lim_{i \to \infty} u_i = -1 \text{ near (not necessarily on) } \{ \theta_2 \} \times M.
\end{equation}
The other cases are ruled out by replacing $u_i$'s with $-u_i$'s.

Let $\delta \in (0, \tfrac{1}{6} \min \{ \theta_2 - \theta_1, \theta_3 - \theta_2 \})$. It follows from the Hausdorff convergence assumption on our nodal sets that, for all large $i$, 
\begin{equation} \label{eq:m.hausdorff.dist}
	\dist_{\operatorname{Haus}}(\{ u_i = 0 \}, \Sigma) < \delta
\end{equation}
and from \eqref{eq:m.limit.neg1}, \eqref{eq:m.hausdorff.dist} it follows that
\begin{equation} \label{eq:m.containment.geq}
	\Omega_i := \{ u_i \geq 0 \} \cap [\theta_1 + \delta, \theta_3 - \delta] \times M \subset (\theta_2 - \delta, \theta_2 + \delta) \times M
\end{equation}
Passing to larger $i$ if necessary, let 
\[ \breve u_i : [\theta_2 - 3 \delta, \theta_2 + 3 \delta] \times M \to [-1, 1] \]
be obtained by applying Lemma \ref{lemm:model} with $\ell = \delta$, $\eps = \eps_i$, translating to the right by $\theta_2$, extending by an odd reflection to the left and the right, and finally extending trivially along $M$. It follows from \eqref{eq:m.containment.geq} and Lemma \ref{lemm:gmn} that
\[ u_i < \breve u_i \text{ on } \{ u_i \geq 0 \} \cup ([\theta_2 - 3\delta, \theta_2 - \delta] \cup [\theta_2 + \delta, \theta_2 + 3\delta]) \times M. \]
Since $u_i < 0 < \breve u_i$ in the remainder of $\operatorname{domain}(\breve u_i)$, this implies
\[ u_i < \breve u_i \text{ on } \operatorname{domain}(\breve u_i). \]
We can now contradict the strong maximum principle by sliding $\breve u_i$ to the left along the $\RR$ factor: due to \eqref{eq:m.hausdorff.dist}, the first touching point has to occur before sliding by  $2\delta$ and cannot occur on the boundary of (the slid version of) $\operatorname{domain}(\breve u_i)$. This yields (a). Note that, in particular, this improves $m > 2$ to $m > 3$.

It remains to prove (b).  We may suppose that
\begin{equation} \label{eq:m.pm.1}
	\lim_{i \to \infty} u_i = -1 \text{ on } (\theta_2, \theta_3) \times M,
\end{equation}
\begin{equation} \label{eq:m.right.longer}
	\theta_4 - \theta_3 > \theta_3 - \theta_2.
\end{equation}
The remaining cases can be ruled out by replacing $u_i$'s with $-u_i$'s and possibly reflecting across the $\RR$ axis. Note also that by \eqref{eq:m.pm.1} and (a), we also have
\begin{equation} \label{eq:m.pm.1.b}
	\lim_{i \to \infty} u_i = 1 \text{ on } ((\theta_1, \theta_2) \cup (\theta_3, \theta_4)) \times M.
\end{equation}

Fix $\delta \in (0, \tfrac14 \min \{ \theta_2 - \theta_1, \theta_3 - \theta_2, \theta_4 - \theta_3, (\theta_4 - \theta_3) - (\theta_3 - \theta_2)\})$, as in possible by \eqref{eq:m.right.longer}. It follows from the Hausdorff convergence assumption on our nodal sets that, for all large $i$, 
\begin{equation} \label{eq:m.hausdorff.dist.2}
	\dist_{\operatorname{Haus}}(\{ u_i = 0 \}, \Sigma) < \delta.
\end{equation}

It follows from \eqref{eq:m.pm.1}, \eqref{eq:m.pm.1.b}, and \eqref{eq:m.hausdorff.dist.2} that for large $i$,
\begin{equation} \label{eq:m.containment.leq}
	\{ u_i \leq 0 \} \cap [\theta_1 + \delta, \theta_4 - \delta] \times M \subset (\theta_2 - \delta, \theta_3 + \delta) \times M.
\end{equation}
Passing to larger $i$ if necessary, let 
\[ \bar u_i : [\theta_2 - \delta, \theta_3 + \delta + (\theta_3 - \theta_2 + 2\delta)] \times M \to [-1, 1] \]
be obtained by applying Lemma \ref{lemm:model} with $\ell = \theta_3 - \theta_2 + 2\delta $, $\eps = \eps_i$, translating to the right by $\theta_3 + \delta$, extending by an odd reflection to the left, and finally extending trivially along $M$.

It follows from \eqref{eq:m.containment.leq} and Lemma \ref{lemm:gmn} that
\[ \bar u_i < u_i \text{ on } \{ u_i \leq 0 \} \cup [\theta_3 + \delta, \theta_3 + \delta + (\theta_3 - \theta_2 + 2\delta)] \times M \]
Since $\bar u_i < 0 < u_i$ in the remainder of $\operatorname{domain}(\bar u_i)$, this implies
\[ \bar u_i < u_i \text{ on } \operatorname{domain}(\bar u_i). \]
We can now contradict the strong maximum principle by sliding $\bar u_i$ to the left along the $\RR$ factor: due to \eqref{eq:m.hausdorff.dist.2}, the first touching point has to occur before sliding by  $2\delta$ and cannot occur on the boundary of (the slid version of) $\operatorname{domain}(\bar u_i)$. This yields (a).

\appendix

\section{The main tools}

\begin{lemm}[{\cite[Proposition 2.2]{Gui}}] \label{lemm:gui}
	Suppose $u$ is a double-well phase transition with parameter $\eps > 0$ on a closed Riemannian manifold. There exist constants $C > 0$, $\kappa > 0$, depending only on the ambient manifold and the double-well potential $W$, such that
	\[ |u^2 - 1| \leq C e^{-\kappa \dist( \cdot, \{ u = 0 \})}. \]
\end{lemm}

The statement of \cite[Proposition 2.2]{Gui} in the reference is for $\mathbf{R}^2$ but the proof is local and adapts in the obvious way to any closed Riemannian manifold.

\begin{lemm}[{\cite[Corollary 7.4]{GuaracoMarquesNeves}}] \label{lemm:gmn}
	Suppose $u, v$ are double-well phase transitions with parameter $\eps > 0$ on a compact connected domain $\Omega$. If $u > 0$ on $\Omega$ and $v = 0$ on $\partial \Omega$, then $u > v$.
\end{lemm}

The statement of \cite[Corollary 7.4]{GuaracoMarquesNeves} in the reference also asks that $\partial \Omega$ be smooth, but this is not used in the proof. The reference also assumed $W(x) = \tfrac14 (1-x^2)^2$, but only the properties assumed in Definition \ref{defi:double.well} were needed in the proof (see \cite[Lemma 7.3]{GuaracoMarquesNeves}).

\textbf{Acknowledgments}. The author was partially supported by NSF DMS-2147521 and thanks Otis Chodosh and Alessandro Pigati for stimulating conversations. The author also thanks the anonymous referee for many helpful comments.

\bibliography{main} 
\bibliographystyle{amsalpha}

\end{document}